\documentclass[12pt]{amsart}
\author{Abbey Bourdon}
\author{Pete L. Clark}
\author{Paul Pollack}
\address{University of Georgia \\ Mathematics Department \\ Boyd Graduate Studies Research Center \\ Athens, GA 30602\\USA}
\email{abourdon@uga.edu}
\email{pete@math.uga.edu}
\email{pollack@uga.edu}
\title{Anatomy of Torsion in the CM Case}
\usepackage{amsmath,amssymb,amsthm,sistyle,booktabs}
\usepackage{geometry}

\usepackage{tikz}
\usepackage{mathtools}

\SIthousandsep{,}
\DeclareMathAlphabet{\curly}{U}{rsfs}{m}{n}
\newtheorem{thm}{Theorem}[section]
\newtheorem{cor}[thm]{Corollary}
\newtheorem{prop}[thm]{Proposition}
\newtheorem{lem}[thm]{Lemma}
\newtheorem{conj}[thm]{Conjecture}

\theoremstyle{definition}
\theoremstyle{remark}
\newtheorem{rmk}{Remark}[section]
\newtheorem{remarks}[rmk]{Remarks}

\AtBeginDocument{\addtocontents{toc}{\protect\setlength{\parskip}{0pt}}}

\begin{document}

\newcommand\leg{\genfrac(){.4pt}{}}
\renewcommand{\labelenumi}{(\roman{enumi})}
\def\Bb{\curly{B}}
\def\ff{\mathfrak{f}}
\def\F{\mathbb{F}}
\def\N{\mathbb{N}}
\def\Q{\mathbb{Q}}
\def\Z{\mathbb{Z}}
\def\R{\mathbb{R}}
\def\O{\mathcal{O}}
\def\aa{\mathfrak{a}}
\def\pp{\mathfrak{p}}
\def\qq{\mathfrak{q}}
\def\Aa{\curly{A}}
\def\Dd{\curly{D}}
\def\Gg{\curly{G}}
\def\Pp{\curly{P}}
\def\Ss{\curly{S}}
\def\End{\mathrm{End}}
\def\M{\mathcal{M}}
\newcommand{\Ok}{\mathcal{O}_K}
\newcommand{\plc}[1]{{\color{red} \sf PLC: [#1]}}
\newcommand{\ppp}[1]{{\color{red} \sf PPP: [#1]}}
\newcommand{\ra}{\rightarrow}
\newcommand{\CM}{\mathrm{CM}}
\newcommand{\ord}{\operatorname{ord}}

%

\newcommand{\tors}{\operatorname{tors}}
\begin{abstract} Let $T_{\mathrm{CM}}(d)$ denote the maximum size of a torsion subgroup of a CM elliptic curve over a degree $d$ number field.  We initiate a systematic study of the asymptotic behavior of $T_{\CM}(d)$ as an ``arithmetic function''.  Whereas a recent result of the last two authors computes the upper
order of $T_{\CM}(d)$, here we determine the lower order, the typical order
and the average order of $T_{\CM}(d)$ as well as study the number of isomorphism
classes of groups $G$ of order $T_{\CM}(d)$ which arise as the torsion subgroup
of a CM elliptic curve over a degree $d$ number field.  To establish these analytic results we need to extend some prior algebraic results.  Especially, if $E_{/F}$ is a CM elliptic curve over a degree $d$ number field,
we show that $d$ is divisible by a certain function of $\# E(F)[\tors]$, and
we give a complete characterization of all degrees $d$ such that every torsion
subgroup of a CM elliptic curve defined over a degree $d$ number field
already occurs over $\Q$.
\end{abstract}
\maketitle
\tableofcontents

\section{Introduction}
\addtocounter{subsection}{-1}
\subsection{Terminology, notation and conventions}\noindent
Throughout, $\ell$ denotes a prime number. We say $\ell^{\alpha}$ \emph{exactly divides} $n$, and write $\ell^{\alpha} \parallel n$, if $\ell^{\alpha} \mid n$ but $\ell^{\alpha+1}\nmid n$. We use the notation $\omega(n)$ for the number of distinct primes dividing $n$, and we write $\Omega(n)$ for the number of primes dividing $n$ counted with multiplicity.

If $K$ is a number field, we let $\O_K$ denote its ring of integers, $\Delta_K$ its discriminant, $h_K$ its class number, and $w_K$ the number of roots of unity lying in $K$. For an ideal $\aa$ of $\O_K$, we denote by $K^{(\aa)}$ the $\aa$-ray class field of $K$.

We say an elliptic curve $E$ over a field of characteristic zero has $\O$-CM if $\End(E) \cong \O$, where $\O$ is an order in an imaginary quadratic field $K$. The statement ``$E$ has $K$-CM'' means that $E$ has $\O$-CM for some order $\O$ in $K$.

The \emph{torsion rank} of a finite abelian group $G$ is the minimal number of elements required to generate $G$.

Let $\Aa$ be a subset of the positive integers.  We define the \emph{upper
density}
\[ \overline{\delta}(\Aa) = \limsup_{x \rightarrow \infty} \frac{ \# \Aa \cap [1,x]}{x} \]
and the \emph{lower density}
\[ \underline{\delta}(\Aa) = \liminf_{x \rightarrow \infty} \frac{ \# \Aa \cap [1,x]}{x} \]
When $\overline{\delta}(\Aa) = \underline{\delta}(\Aa)$, we denote the common
quantity by $\delta(\Aa)$ and call it the \emph{asymptotic density} of $\Aa$.

\subsection{$T(d)$ versus $T_{\CM}(d)$}\noindent A celebrated theorem of L. Merel \cite{merel96} asserts that if $E$ is an elliptic curve defined over a degree $d$ number field $F$, then $\#E(F)[\textrm{tors}]$ is bounded by a constant depending only on $d$.  The best known bounds, due to J. Oesterl\'e (unpublished) and P. Parent \cite{parent99}, show that the prime powers appearing in the exponent of $E(F)[\textrm{tors}]$ are bounded by quantities which are exponential $d$.

For certain classes of curves one can do much better.  When the $j$-invariant of $E$ is an algebraic integer, Hindry and Silverman \cite{HS99} showed that for $d\ge 2$,
\[ \#E(F)[\textrm{tors}] \le 1977408 d\log{d}. \]
Under the stronger assumption that $E$ has complex multiplication (CM), it has recently been shown \cite{CP15} that there is an effectively computable $C > 0$ such that
\begin{equation}
\label{PETEPAULEQ}
\forall d \geq 3, \enskip  \#E(F)[\textrm{tors}] \le C d\log\log{d}.
\end{equation}
Let $T_{\mathrm{CM}}(d)$ denote the largest size of a torsion subgroup
of a CM elliptic curve defined over a number field of degree $d$.  Combining
(\ref{PETEPAULEQ}) with work of Breuer \cite{breuer10} gives
\begin{equation}
\label{PETEPAULEQ2}
 \limsup_{d \ra \infty} \frac{T_{\CM}(d)}{d \log \log d} \in (0,\infty).
\end{equation}
In particular (\ref{PETEPAULEQ}) is \emph{sharp} up to the value of $C$.

Let $T(d)$ be the largest size of a torsion subgroup of an elliptic curve
over a degree $d$ number field, and let $T_{\neg \CM}(d)$ be the largest size of the torsion subgroup of an elliptic curve \emph{without} complex multiplication
over a degree $d$ number field, so $T(d) = \max\{T_{\CM}(d),\ T_{\neg
\CM}(d)$\}.
We are far from knowing the truth about $T_{\CM}(d)$ but we expect --- cf.
\cite[$\S 1$]{CP15} --- that $T_{\neg \CM}(d) = O(\sqrt{d \log \log d})$.
Again Breuer's work provides
lower bounds to show that such an upper bound would be sharp up to a constant.
This would also imply that $T(d) = T_{\mathrm{CM}}(d)$ for infinitely many $d$.

It is not yet known whether $T(d) = T_{\mathrm{CM}}(d)$ for any $d \in \Z^+$.  We have  \cite{Mazur77, Sutherland12}
\[ T_{\mathrm{CM}}(1) = 6 < 16 = T(1), \enskip T_{\mathrm{CM}}(2) = 12 < 24 = T(2) . \]
Since these are the only known values of $T(d)$, finding values of $d$
for which $T(d) = T_{\mathrm{CM}}(d)$ seems beyond reach.  But $T_{\CM}(d)$ is known for infinitely many values, so we can find values of $d$ for which $T(d) > T_{\CM}(d)$.  Especially, by \cite[Theorem 1.4]{BCS15} we have
\[ \text{For all primes } p \geq 7, \ T_{\mathrm{CM}}(p) = 6 < 16 = T(1) \leq T(p). \]
Moreover, from \cite{TORS2} we know $T_{\CM}(d)$ for all $d \leq 13$, which
presents the prospect of showing $T(d) > T_{\CM}(d)$ for some further small
values of $d$ simply by exhibiting a non-CM elliptic curve in degree $d$
with large enough torsion subgroup.  We make use of the following recent
computational results:
\begin{itemize}
\item Najman \cite{Najman14}: $T(3) \geq 21$.
\item Jeon--Kim--Park \cite{JKP06}: $T(4) \geq 36$.
\item van Hoeij \cite{vanHoeij14}: $T(5) \geq 30$, $T(6) \geq 37$, $T(9)
\geq 34$.
\end{itemize}
Combining with the calculations of \cite{TORS2} we find:
\[ \forall d \in \{3,4,5,6,9\}, \enskip T(d) > T_{\CM}(d). \]
On the other hand, we have $T_{\mathrm{CM}}(8) = T_{\mathrm{CM}}(10) = 50$,
$T_{\mathrm{CM}}(12) = 84$, and there are no known non-CM elliptic curves with larger torsion subgroups in these degrees.  In degree $8$ the largest order
of a torsion point on a CM elliptic curve is $39$, whereas there is a point
of order $50$ on a non-CM elliptic curve in degree $8$.  However there is a
point of order $50$ on a CM elliptic curve of degree $10$, and $50$ is
the largest value of $N$ for which the tables in \cite{vanHoeij14} record a degree $10$ point on $Y_0(N)$.  Further comparison of the tables of
\cite{vanHoeij14} to the work of \cite{CCS13} and \cite{TORS2} gives several values of $N$ for which the smallest known degree of a point on $Y_1(N)$
is attained by a CM-point, e.g. $N \in \{57, 61,67,73,79\}$.

In summary, it seems that the tools are not yet available to determine $T(d)$ for more than a few values of $d$, let alone to arrive at a theoretical
understanding of the asymptotic behavior of this function.  Henceforth we
consider only the CM case, which is much more tractable and apparently
related to the non-CM case in interesting ways.

\subsection{Anatomy of $T_{\CM}(d)$}
\noindent
The goal of the present paper is to regard $T_{\mathrm{CM}}(d)$ as an ``arithmetic function'' and study its behavior for large values of $d$ in
the fashion that one studies functions like Euler's totient function $\varphi$.
From this perspective, (\ref{PETEPAULEQ2}) gives the \emph{upper order} of $T_{\mathrm{CM}}(d)$.  However, as with more classical arithmetic
functions, $T_{\mathrm{CM}}(d)$ exhibits considerable variation, and it is also interesting to ask about its lower order, its average order, and its ``typical order'' (roughly, its behavior away from a set of $d$ of small density).  It turns out that now is the right time to address these questions: by using --- and, in some cases, sharpening --- the results of \cite{BCS15} and \cite{CP15}, we find that we have enough information on the elliptic curve theory side
to transport these questions into the realm of elementary/analytic number theory and then answer them.

We first determine the typical order (in a reasonable sense) of $T_{{\rm CM}}(d)$.

\begin{thm}\label{thm:density}\mbox{ }
\begin{enumerate}
\item[(i)] For all $\epsilon > 0$, there is a
positive integer $B_{\epsilon}$ such that
\[ \overline{\delta}( \{d \in \Z^+ \mid T_{\CM}(d) \geq B_{\epsilon} \}) \leq \epsilon. \]
\item[(ii)] For all $B \in \Z^+$, we have
\[ \underline{\delta} ( \{d \in \Z^+ \mid T_{\CM}(d) \geq B \}) > 0. \]
\end{enumerate}
\end{thm}

Though stated separately for parallelism, the proof of
Theorem \ref{thm:density}(ii) is immediate.  Indeed, starting with any CM elliptic curve $E/\Q$, we may adjoin the coordinates of a point of order $N$ to obtain a field $F_0$ of degree $d_0$ (say). Considering extensions of $F_0$, we find that $T_{\mathrm{CM}}(d) \ge N$ whenever $d_0 \mid d$ and thus
\[ \underline{\delta} ( \{d \in \Z^+ \mid T_{\CM}(d) \geq B \}) \geq \frac{1}{d_0}. \]


\noindent
We turn next to the average order of $T_{{\rm CM}}(d)$.

\begin{thm}\label{thm:avg0}\mbox{ }
\begin{enumerate}
\item[(i)] We have $ \frac{1}{x} \sum_{d \le x}T_{{\rm CM}}(d) = x/(\log{x})^{1+o(1)}$.  In other words: for all $c < 1$,
\[ \lim_{x \ra \infty} \frac{ \frac{1}{x} \sum_{d \leq x} T_{\CM}(d)}{x/\log^c x} = 0, \]
and for all $C > 1$ we have
\[ \lim_{x \ra \infty} \frac{ \frac{1}{x} \sum_{d \leq x} T_{\CM}(d)}{x/\log^C x} = \infty. \]

\item[(ii)]
We have $\frac{1}{x}\sum_{\substack{d \le x\\ 2\nmid d}} T_{{\rm CM}}(d) = x^{1/3 + o(1)}$.  In other words: for all $c < \frac{1}{3}$,
\[ \lim_{x \ra \infty} \frac{\frac{1}{x}\sum_{\substack{d \le x\\ 2\nmid d}} T_{{\rm CM}}(d)}{x^c} = \infty, \]
and for all $C > \frac{1}{3}$,
 \[ \lim_{x \ra \infty} \frac{\frac{1}{x}\sum_{\substack{d \le x\\ 2\nmid d}} T_{{\rm CM}}(d)}{x^C} = 0. \]

\end{enumerate}
\end{thm}

\begin{remarks}\mbox{ }
\begin{enumerate}
\item[(i)] The average order of $T_{{\rm CM}}(d)$ restricted to odd degrees is considerably smaller than its average order restricted to even
degrees. This is another confirming instance of the odd/even dichotomy explored in \cite{BCS15}.
\item[(ii)] The average order of $T_{{\rm CM}}(d)$ is considerably larger than the conjectural maximal order $\sqrt{d \log \log d}$ of $T(d)$.
\end{enumerate}
\end{remarks}
\noindent
Now we turn to the \emph{lower order} of $T_{\mathrm{CM}}(d)$.  When $E$ is a CM elliptic curve over $\Q$, Olson \cite{olson74} showed that there are precisely six possibilities for the group $E(\Q)[{\rm tors}]$ (up to isomorphism): the trivial group $\{\bullet\}$, $\Z/2\Z$, $\Z/3\Z$, $\Z/4\Z$, $\Z/6\Z$, and $\Z/2\Z \times \Z/2\Z$.  We call these the \emph{Olson groups}. From \cite[Theorem 2.1(a)]{BCS15} we know that for any abelian variety defined
over a number field $A_{/F}$ and all integers $d \geq 2$, there are infinitely
many degree $d$ extensions $L/F$ with $A(L)[\tors] = A(F)[\tors]$.  In particular, since the Olson groups occur over $\Q$, each of them occurs as the torsion subgroup of a CM elliptic curves over a number field of every degree, and thus $T_{\mathrm{CM}}(d) \geq 6$ for all $d$.  Let us say that $d \in \Z^+$ is an \emph{Olson degree} if the only torsion subgroups of CM elliptic curves in
degree $d$ are Olson groups.  In \cite[Theorem 1.4]{BCS15} it was shown
that every prime number $d \geq 7$ is an Olson degree.  We deduce
\[
\liminf_{d \ra \infty} T_{\mathrm{CM}}(d) = 6.
\]

\begin{rmk}
If $d$ is an Olson degree, then $T_{\CM}(d) = 6$.  In fact the
converse holds, so the Olson degrees are precisely the degrees at which
$T_{\CM}(d)$ attains its minimum value.  This comes down to showing that if $T_{\CM}(d) = 6$, then
there is no CM elliptic curve $E$ defined over a degree $d$ number field $F$
with an $F$-rational point of order $5$.  But from \cite[Theorem 1.5]{BCS15},
the existence of such an $E_{/F}$ forces $d$ to be even, and thus
$T_{\CM}(d) \geq T_{\CM}(2) = 12$.
\end{rmk}

It is natural to ask for more precise information about the Olson degrees.  Above we saw that the upper order of $T_{{\rm CM}}(d)$ is attained (or even approached) only on a very small set of $d$'s.  The result that all prime degrees $d \geq 7$ are Olson leaves open the possibility
that the set of Olson degrees has density zero.  In fact this is not the case.

\begin{thm}\label{thm:density2} The set of Olson degrees has positive asymptotic density.
\end{thm}
We also extend \cite[Theorem 1.4]{BCS15} in the following complementary direction.
\begin{thm}\label{thm:primepower}  For all $n \in \Z^+$, there is a $P = P(n)$
such that for all primes $p \geq P$,  the number $p^n$ is an Olson degree.
\end{thm}
Finally we consider the distribution of groups $G$ that realize the maximality of $T_{\rm CM}(d)$. Say that the finite abelian group $G$ is a \emph{maximal torsion subgroup in degree $d$} if $\#G = T_{\rm CM}(d)$ and there is a CM elliptic curve $E$ over a degree $d$ number field $F$ with $E(F)[{\rm tors}] \cong G$. From the maximal order result in \cite{CP15}, each maximal torsion subgroup  $G$ in degree $d \le x$ has size $O(x\log\log{x})$. In view of Lemma \ref{lem:tau0mean} below, this leaves us with $\asymp x\log\log{x}$ possibilities for $G$.  The next result describes how many such groups actually occur.

\begin{thm}\label{thm:distinctgroups} For $d \in \Z^+$, let $\M(d)$
be the set of isomorphism classes of groups $G$ such that $\#G = T_{\CM}(d)$
and $G \cong E(F)$ for a CM elliptic curve $E$ defined over a degree $d$
number field $F$.  Then
\[ \# \bigcup_{d \leq x} \M(d) = x/(\log{x})^{1+o(1)}. \]
\end{thm}

\subsection{Algebraic results}\noindent In order to prove the results of the last section we need to sharpen and extend some of the algebraic results of \cite{CCS13} and \cite{BCS15}.

The prototypical result that gives leverage on torsion in the CM case
is the following theorem of Silverberg and Prasad-Yogananda \cite{silverberg88,PY01}: if $E_{/F}$
is an $\O$-CM elliptic curve defined over a number field $F$ admitting
an $F$-rational point of order $N$, then
\[ \varphi(N) \leq \# \O^{\times} [F:\Q]. \]
Moreover, if $F \supset K$ then
\[ 2\varphi(N) \leq \# \O^{\times} [F:\Q], \]
whereas if $F \not\supset K$ then
\[ \varphi(\# E(F)[\tors]) \leq \# \O^{\times} [F:\Q].\]
We call these inequalities the \emph{SPY bounds}.  They were
refined when $N$ is prime in \cite{CCS13} and \cite{BCS15} by separate consideration of the cases in which $N$ is split, inert or ramified in the CM field $K$.  Moreover, at least in the case of
CM by the maximal order, classical theory gives a tight relationship between
$F$-rational torsion and the containment in $F$ of ray class fields of $K$.
The following result systematically relates SPY-type bounds, for prime powers $N$, to ray class containments.

\begin{thm}
\label{BIGSPY}
Let $F$ be a degree $d$ number field containing an imaginary quadratic field $K$. Let $E_{/F}$ be an elliptic curve with $\O$-CM, where $\O$ is the order in $K$ of discriminant $\Delta$. Suppose $E(F)[{\ell^{\infty}}] \cong \Z/\ell^{a}\Z \times \Z/\ell^{b}\Z$, where $b \ge a \ge 0$ and $b\ge 1$. Then:
\begin{enumerate}
\item If $\leg{\Delta}{\ell}=-1$, then $a=b$, and $\ell^{2b-2} (\ell^2-1) \mid w_K \cdot [F \cap K^{(\ell^{b} \O_K)}: K^{(\O_K)}]$.
\item If $\leg{\Delta}{\ell}=1$ and $a=0$, then $\ell^{b-1}(\ell-1) \mid w_K \cdot [F \cap K^{(\ell^{b} \O_K)}: K^{(\O_K)}]$.
\item If $\leg{\Delta}{\ell}=1$ and $a \geq 1$, then $\ell^{a+b-2}(\ell-1)^2 \mid w_K \cdot [F \cap K^{(\ell^{b} \O_K)}: K^{(\O_K)}]$.
\item If $\leg{\Delta}{\ell}=0$ and $\ell$ ramifies in $K$, then $\ell^{a+b-1}(\ell-1) \mid w_K \cdot [F \cap K^{(\ell^{b} \O_K)}: K^{(\O_K)}]$.
\item If $\leg{\Delta}{\ell}=0$ and $\ell$ is unramified in $K$, then \newline$\ell^{\max\{a+b-2,0\}}(\ell-1)(\ell-\leg{\Delta_K}{\ell}) \mid w_K \cdot [F \cap K^{(\ell^{b} \O_K)}: K^{(\O_K)}]$.
\end{enumerate}
\end{thm}
\noindent
These divisibility results combine in a natural way if one wants to consider the full group of $F$-rational torsion (see Theorem \ref{lem:div2}).

The other main algebraic result is a complete determination of all Olson degrees.  Recall that a set of $\Aa$ of positive integers is called a \emph{set of multiples} if whenever $a \in \Aa$, every multiple of $\Aa$ is also in $\Aa$. This is easily seen to be equivalent to requiring that $\Aa = M(\Gg)$ for some set of positive integers $\Gg$, where
\[ M(\Gg) = \{n \in \Z^{+} : g\mid n \text{ for some $g\in \Gg$}\}. \]
We call $\Gg$ a set of \emph{generators} for $\Aa$.

\begin{table}
\begin{tabular}{c|c}
$N$ & \# Olson degrees in $[1,N]$\\
\midrule\midrule
\num{1000} & \num{265} \\
\num{10000} & \num{2649} \\
\num{100000} & \num{26474} \\
\num{1000000} & \num{264633} \\
\num{10000000} & \num{2646355} \\
\num{100000000} & \num{26462845} \\
\num{1000000000} & \num{264625698} \\
\num{10000000000} & \num{2646246218} \\
\num{100000000000} & \num{26462418808}
\end{tabular}
\vskip 0.07in
\caption{Counts of Olson degrees \label{tbl:olsoncounts} to $10^{11}$.}
\end{table}

\begin{thm}
\label{OLSONDEGREETHM}
The set of non-Olson degrees can be written as $M(\Gg)$, where \[ \Gg = \{2\} \cup \left\{\frac{\ell-1}{2} \cdot h_{\Q(\sqrt{-\ell})} \mid \ell \equiv 3\pmod{4},~\ell > 3\right\}. \]
\end{thm}
\noindent
An algorithm for computing all torsion subgroups of CM elliptic curves in degree
$d$ is presented in \cite{TORS2}.  In principle this algorithm allows
us to determine whether a given degree $d$ is Olson.  However, the algorithm
requires as input the list of all imaginary quadratic fields of class
number properly dividing $d$ so is for sufficiently large composite $d$ quite impractical.
In contrast, using Theorem \ref{OLSONDEGREETHM}, one can compute in a day on a modern desktop computer that there are
\num{26462418808} Olson degrees $d \le 10^{11}$.  Since $\pi(10^{11}) = \num{4118054813}$, this adds
\num{22344363994} composite values of $d$ for which the complete list of torsion
subgroups of CM elliptic curves in degree $d$ is known. Such calculations
suggest that the density of Olson degrees, which by Theorem
\ref{thm:density2} lies in $(0,1)$, is in fact slightly larger than $\frac{1}{4}$; see Table \ref{tbl:olsoncounts}.

We also found that for all primes $p > 5$ and all $n \in \Z^+$,
if $p^n \leq 10^{30}$ then $p^n$ is an Olson degree.\footnote{A warning: To perform the above computations, we made extensive use of the \texttt{PARI/GP} command \texttt{quadclassunit} to compute class numbers of imaginary quadratic fields. That algorithm has been proved correct \emph{only under the assumption of the Generalized Riemann Hypothesis}. However, the counts up to $10^6$ in Table \ref{tbl:olsoncounts} have been certified unconditionally, as has the result that there are no non-Olson prime powers $p^n \le 10^{14}$ (with $p> 5$).}
Thus we conjecture
the following strengthening of Theorem \ref{thm:primepower}.

\begin{conj}\label{conj:bob} $p^n$ is an Olson degree for every prime $p > 5$ and all $n \in \Z^+$.
\end{conj}

\section{Divisibility requirements for rational torsion}
\noindent
The next two results are taken from the already mentioned work \cite{CP15}.

\begin{lem}[{\cite[Theorem 5]{CP15}}]\label{lem:rayclass} Let $K$ be an imaginary quadratic field, $F\supset K$ be a number field, $E_{/F}$ a $K$-CM elliptic curve, and $N \in \Z^{+}$. If $(\Z/N\Z)^{2} \hookrightarrow E(F)$, then $F \supset K^{(N \O_K)}$.
\end{lem}

\begin{lem}[{\cite[Theorem 6]{CP15}}]\label{lem:squaring} Let $K$ be an imaginary quadratic field, $F\supset K$  a number field, and $E_{/F}$  an $\O$-CM elliptic curve. Suppose that $E(F)[\ell^{\infty}] \cong \Z/{\ell^a}\Z \times \Z/\ell^{b}\Z$, where $b\ge a \ge 0$ and $b \ge 1$. Then $[F(E[\ell^b]): F] \le \ell^{b-a}$. In fact, letting $\Delta$ denote the discriminant of $\O$, we have the following more precise results:
\begin{enumerate}
\item If $\leg{\Delta}{\ell} =0$ or $=-1$, then $[F(E[\ell^b]):F]\mid \ell^{b-a}$.
\item If $\leg{\Delta}{\ell}=1$, then either $a=0$ and $[F(E[\ell^b]):F] \mid (\ell-1)\ell^{b-1}$, or $a>0$ and $[F(E[\ell^b]):F] \mid \ell^{b-a}$.
\end{enumerate}
\end{lem}

\begin{rmk} Statements (i) and (ii) are not explicitly included in \cite[Theorem 6]{CP15}; however, they follow immediately from the proof. In fact, as we recall below, when $\leg{\Delta}{\ell}=-1$ we always have $b=a$.
\end{rmk}

\begin{lem}\label{lem:div} Let $F$ be a degree $d$ number field containing an imaginary quadratic field $K$. Let $E_{/F}$ be an elliptic curve with $\O$-CM, where $\O$ is the order in $K$ of discriminant $\Delta$. Suppose $E(F)[{\ell^{\infty}}] \cong \Z/\ell^{a}\Z \times \Z/\ell^{b}\Z$, where $b \ge a \ge 0$ and $b\ge 1$. If
\begin{enumerate}
\item $\leg{\Delta}{\ell}=-1$, then $a=b$, and $h_K \cdot \ell^{2b-2} (\ell^2-1) \mid w_K \frac{d}{2}$,
\item $\leg{\Delta}{\ell}=1$ and $a=0$, then $h_K \cdot \ell^{b-1}(\ell-1) \mid w_K \frac{d}{2}$,
\item $\leg{\Delta}{\ell}=1$ and $a>0$, then $h_K \cdot \ell^{a+b-2}(\ell-1)^2 \mid w_K \frac{d}{2}$,
\item $\leg{\Delta}{\ell}=0$ and $\ell$ ramifies in $K$, then $h_K \cdot \ell^{a+b-1}(\ell-1) \mid w_K \frac{d}{2}$,
\item $\leg{\Delta}{\ell}=0$ and $\ell$ is unramified in $K$, then $h_K \cdot \ell^{\max\{a+b-2,0\}}(\ell-1)(\ell-\leg{\Delta_K}{\ell}) \mid w_K \frac{d}{2}$.
\end{enumerate}
\end{lem}
\begin{proof} We follow the proof of \cite[Theorem 4.6]{BCS15}. By Lemma \ref{lem:rayclass}, $K^{(\ell^{b} \O_K)} \subset F(E[\ell^b])$.  Recalling that $K(j(E))$ is a ring class field of $K$, we see that $F \supset K(j(E)) \supset K^{(\O_K)}$.  Let $d_0 = [F(E[\ell^b]):F]$.
\begin{figure}
\begin{center}
\begin{tikzpicture}[node distance=2cm]
\node (Q)                  {$\mathbb{Q}$};
\node (K) [above of=Q, node distance=1cm] {$K$};
\node (Kj) [above of=K, node distance=1.5cm] {$K^{(\O_K)}$};
\node (Kl) [above of =Kj, node distance=2 cm] {$K^{(\ell^b \O_K)}$};
\node (FK)  [above right of=Kj, node distance=2.5 cm]   {$F$};
\node (Ftor) [above of =FK, node distance=1.8cm] {$F(E[\ell^b])$};

 \draw[-] (Kj) edge node[right] {$\frac{d}{2h_K}$} (FK);
 \draw[-] (Q) edge node[left] {2} (K);
 \draw[-] (FK) edge node[right] {$d_0$} (Ftor);
 \draw[-] (Kj) edge node[left] {$\frac{\ell^{2b-2}(\ell-1)(\ell-\leg{\Delta_K}{\ell})}{[U:U_{\ell^b}]}$} (Kl);
 \draw[-] (Kj) edge node[left] {$h_K$} (K);
 \draw (Kl) -- (Ftor);

\end{tikzpicture}
\end{center}
\caption{Diagram of fields appearing in the proof of Lemma \ref{lem:div}.}\label{fig:fields}
\end{figure}
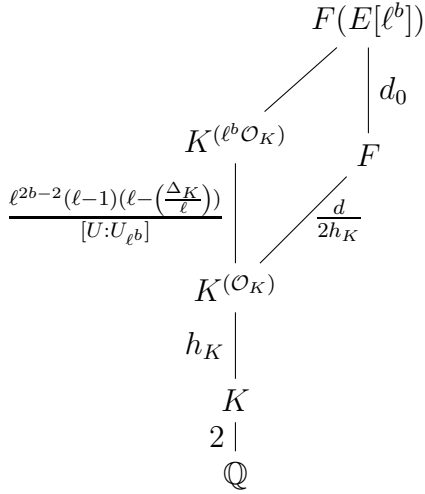

The Hilbert class field $K^{(\O_K)}$ has degree $h_K$ over $K$. From \cite[Proposition 2.1, p. 50]{childress09}, the degree of $K^{(\ell^b \O_K)}$ over $K^{(\O_K)}$ is  $\frac{\Phi(\ell^{b})}{[U: U_{\ell^{b}}]}$. Here $\Phi$ is the analogue of Euler's function for the ideals of $\O_K$, so that
\[ \Phi(\ell^b) = \# (\O_K/\ell^b \O_K)^{\times} = \ell^{2b-2} (\ell-1)(\ell-\leg{\Delta_K}{\ell}),\] $U=\O_K^{\times}$, and $U_{\ell^{b}}$ is the subgroup of units congruent to $1\pmod{\ell^{b}}$. Since $[U: U_{\ell^b}]$ divides $w_K$,
\[ \ell^{2b-2} (\ell-1)(\ell-\leg{\Delta_K}{\ell}) \mid w_K \cdot [F(E[\ell^b]):K^{(\O_K)}] = w_K \frac{d}{2h_K} d_0. \]
Thus,
\begin{equation}\label{eq:nearfinal} \frac{\ell^{2b-2} (\ell-1)(\ell-\leg{\Delta_K}{\ell})}{\gcd(\ell^{2b-2} (\ell-1)(\ell-\leg{\Delta_K}{\ell}),d_0)} \mid w_K \frac{d}{2h_K}. \end{equation}

Suppose that $\leg{\Delta}{\ell}=-1$. In this case, the existence of a single $F$-rational point of order $\ell^b$ implies that $E(F)$ contains $E[\ell^b]$. Indeed, as shown in the proof of \cite[Theorem 4.8]{BCS15}, any torsion point of order $\ell^b$ generates $E[\ell^b]$ as an $\O$-module. Thus, $a=b$ and $d_0=1$, and we obtain the first possibility in the lemma statement.

Suppose next that $\leg{\Delta}{\ell}=1$ and $a=0$. Lemma \ref{lem:squaring} shows that $d_0 \mid \ell^{b-1}(\ell-1)$, so that the left-hand side of \eqref{eq:nearfinal} is divisible by $\ell^{b-1}(\ell-1)$.  Thus, we have the second possibility indicated in the lemma. If $\leg{\Delta}{\ell}=1$ and $a > 0$, then $d_0 \mid \ell^{b-a}$, and the left-hand side of \eqref{eq:nearfinal} is divisible by  $\ell^{a+b-2}(\ell-1)^2$. This gives the third possibility indicated in the lemma statement.

Finally, suppose that $\leg{\Delta}{\ell}=0$. If $\ell$ ramifies in $K$, we use that $d_0 \mid \ell^{b-a}$ to deduce that the left-hand side of \eqref{eq:nearfinal} is divisible by $\ell^{a+b-1} (\ell-1)$. If $\ell$ is unramified in $K$, we use that the denominator in \eqref{eq:nearfinal} divides $\ell^{\min\{b-a,2b-2\}}$ to deduce that the left-hand side of \eqref{eq:nearfinal} is divisible by $\ell^{\max\{a+b-2,0\}} (\ell-1)(\ell-\leg{\Delta_K}{\ell})$. In this way, we obtain the fourth and fifth possibilities in the lemma statement.\end{proof}

\begin{proof}[Proof of Theorem \ref{BIGSPY}] Note that $[FK^{(\ell^b \O_K)}:F]=[K^{(\ell^b \O_K)}:F\cap K^{(\ell^b \O_K)}]$, and that this common value divides both $[F(E[\ell^b]):F]=d_0$ and $[K^{(\ell^b \O_K)}:K^{(\O_K)}] = \Phi(\ell^b)/[U:U_{\ell^b}]$. Consequently, $[K^{(\ell^b \O_K)}:F \cap K^{(\ell^b \O_K)}] \mid \gcd(\Phi(\ell^b), d_0)$,
and so
\[ [K^{(\ell^b \O_K)}: K^{(\O_K)}] \mid \gcd(\Phi(\ell^b), d_0) \cdot [F \cap K^{(\ell^b \O_K)}: K^{(\O_K)}]. \]
Multiply through by $[U: U_{\ell^b}]$ to find that
\[ \frac{\Phi(\ell^b)}{\gcd(\Phi(\ell^b), d_0)} \mid [U:U_{\ell^b}] \cdot [F\cap K^{(\ell^b \O_K)}: K^{(\O_K)}] \mid w_K \cdot [F\cap K^{(\ell^b \O_K)}: K^{(\O_K)}]. \]
But the first term on the left coincides with the left-hand side of \eqref{eq:nearfinal}. The theorem now follows from the case-by-case analysis found in the proof of Lemma \ref{lem:div}.
\end{proof}

Thus far we have examined the divisibility requirements for rational torsion prime-by-prime. However, the conditions combine in a natural way to give divisibility results for the full group of rational torsion. Let $F$ be a number field containing an imaginary quadratic field $K$, and let $E_{/F}$ be an elliptic curve with CM by an order in $K$ of discriminant $\Delta$. Suppose $\#E(F)[{\rm tors}]=n$. For each $\ell \mid n$, we have $E(F)[\ell^{\infty}] \cong \Z/\ell^{a_{\ell}}\Z\times \Z/\ell^{b_{\ell}}\Z$, where $b_{\ell} \ge a_{\ell} \ge 0$ and $b_{\ell} \ge 1$. Thus, $\ell^{\alpha_{\ell}} \parallel n$, where $\alpha_{\ell}\coloneqq a_{\ell} + b_{\ell}$. For each $\ell^{\alpha_{\ell}}$, we define a constant $\lambda_{\ell^{\alpha_{\ell}}}$ in the following way:

\begin{enumerate}
\item If $\leg{\Delta}{\ell}=-1$, then $\lambda_{\ell^{\alpha_{\ell}}}\coloneqq \ell^{2b_{\ell}-2}(\ell^2-1)$.
\item If $\leg{\Delta}{\ell}=1$ and $a_{\ell}=0$, then $\lambda_{\ell^{\alpha_{\ell}}}\coloneqq \ell^{b_{\ell}-1}(\ell-1)$.
\item If $\leg{\Delta}{\ell}=1$ and $a_{\ell} \geq 1$, then $\lambda_{\ell^{\alpha_{\ell}}} \coloneqq \ell^{a_{\ell}+b_{\ell}-2}(\ell-1)^2$.
\item If $\leg{\Delta}{\ell}=0$ and $\ell$ ramifies in $K$, then $\lambda_{\ell^{\alpha_{\ell}}} \coloneqq \ell^{a_{\ell}+b_{\ell}-1}(\ell-1)$.
\item If $\leg{\Delta}{\ell}=0$ and $\ell$ is unramified in $K$, then $\lambda_{\ell^{\alpha_{\ell}}} \coloneqq \ell^{\max\{a_{\ell}+b_{\ell}-2,0\}}(\ell-1)(\ell-\leg{\Delta_K}{\ell})$.
\end{enumerate}
Note that by Theorem \ref{BIGSPY}, we have $\lambda_{\ell^{\alpha_{\ell}}} \mid w_K \cdot [F \cap K^{(\ell^{b_{\ell}} \O_K)}: K^{(\O_K)}]$.

\begin{thm}\label{lem:div2} Suppose that there is a $K$-CM elliptic curve $E$ over a degree $d$ number field $F\supset K$  with $\#E(F)[{\rm tors}]=n$. Then
$h_K \cdot \prod_{\ell \mid n} \lambda_{\ell^{\alpha_{\ell}}}  \mid 6d$.
\end{thm}

\begin{proof} Take any $K$-CM elliptic curve $E_{/F}$ with $[F:\Q]=d$ and $\#E(F)[{\rm tors}]=n$. Let $\O$ be the CM order, and say $\Delta$ is the discriminant of $\O$. As above, for each $\ell \mid n$, write $E(F)[\ell^{\infty}] \cong \Z/\ell^{a_{\ell}}\Z\times \Z/\ell^{b_{\ell}}\Z$, where $b_{\ell} \ge a_{\ell} \ge 0$ and $b_{\ell} \ge 1$. Let $N$ be the exponent of $E(F)[{\rm tors}]$, so that $N = \prod_{\ell \mid n} \ell^{b_{\ell}}$. Let $d_{0,\ell}$ denote the degree $[F(E[\ell^{b_{\ell}}]):F]$, and observe that the degree $d_0$ of $F(E[N])/F$ satisfies
\[ d_0 \mid \prod_{\ell \mid n} d_{0,\ell}. \]
Using that $F(E[N]) \supset K^{(N\O_K)}$,  we find that
\begin{equation}\label{eq:bigdiv} \prod_{\ell \mid n} \ell^{2b_{\ell}-2}(\ell-1)(\ell-\leg{\Delta_K}{\ell}) = [U:U_{N}] \cdot [K^{(N \O_K)}: K^{(\O_K)}] \mid w_K \frac{d}{2h_K} d_0 \mid w_K\frac{d}{2 h_K} \prod_{\ell \mid n}d_{0,\ell}. \end{equation}

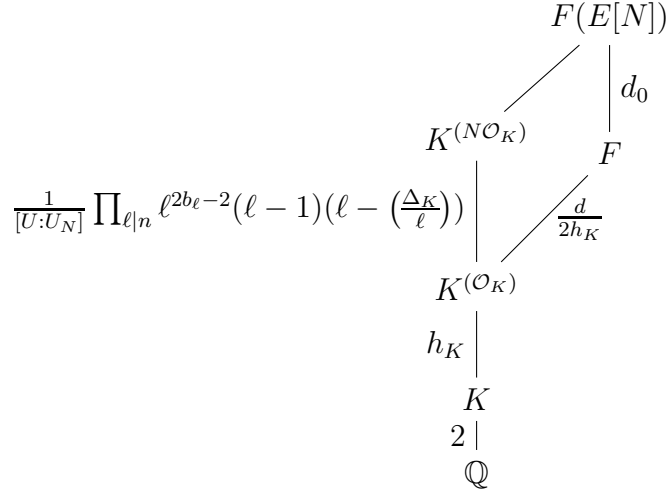
\begin{figure}
\begin{center}
\begin{tikzpicture}[node distance=2cm]
\node (Q)                  {$\mathbb{Q}$};
\node (K) [above of=Q, node distance=1cm] {$K$};
\node (Kj) [above of=K, node distance=1.5cm] {$K^{(\O_K)}$};
\node (Kl) [above of =Kj, node distance=2 cm] {$K^{(N \O_K)}$};
\node (FK)  [above right of=Kj, node distance=2.5 cm]   {$F$};
\node (Ftor) [above of =FK, node distance=1.8cm] {$F(E[N])$};

 \draw[-] (Kj) edge node[right] {$\frac{d}{2h_K}$} (FK);

 \draw[-] (Q) edge node[left] {2} (K);
 \draw[-] (FK) edge node[right] {$d_0$} (Ftor);
 \draw[-] (Kj) edge node[left] {$\frac{1}{{[U:U_{N}]}}\prod_{\ell \mid n} \ell^{2b_{\ell}-2}(\ell-1)(\ell-\leg{\Delta_K}{\ell})$} (Kl);
 \draw[-] (Kj) edge node[left] {$h_K$} (K);
 \draw (Kl) -- (Ftor);

\end{tikzpicture}
\end{center}
\caption{Diagram of fields appearing in the proof of Theorem \ref{lem:div2}.}\label{fig:fields2}
\end{figure}
\noindent
Suppose first that $\alpha_{\ell}\coloneqq a_{\ell}+b_{\ell} \ge 2$. Then the case analysis in the proof of Lemma \ref{lem:div} shows that $d_{0,\ell} \mid \ell^{2b_{\ell}-2}(\ell-1)(\ell-\leg{\Delta_K}{\ell})$, and that the quotient $\ell^{2b_{\ell}-2}(\ell-1)(\ell-\leg{\Delta_K}{\ell})/d_{0,\ell}$ is a multiple of $\lambda_{\ell^{\alpha_{\ell}}}$.

Now suppose that $\alpha_{\ell}=1$. Then $a_{\ell}=0$ and $b_{\ell}=1$. Note that we cannot have $\leg{\Delta}{\ell}=-1$ in this case, since that condition forces $a_{\ell}=b_{\ell}$. If $\leg{\Delta}{\ell}=1$, then $d_{0,\ell} \mid \ell-1$, and so
\begin{equation}\label{eq:eithercase} \lambda_{\ell}=\ell-1 \mid \ell^{2b_{\ell}-2} (\ell-1)(\ell-\leg{\Delta_K}{\ell})/d_{0,\ell}. \end{equation}
If $\leg{\Delta}{\ell}=0$ and $\leg{\Delta_K}{\ell}=0$, then $d_{0,\ell} \mid \ell$, so that again \eqref{eq:eithercase} holds.  Note that if $\leg{\Delta}{\ell} = 0$ but $\leg{\Delta_K}{\ell} \ne 0$, then $d_{0,\ell} \mid \ell$ while

\[ \lambda_{\ell}=\ell^{2b_{\ell}-2}(\ell-1)(\ell-\leg{\Delta_K}{\ell}) \in \{\ell^2-1,(\ell-1)^2\}. \]

Let $\Ss_1$ be the set of prime powers $\ell^{\alpha_\ell}$ exactly dividing $n$ for which either $\alpha_\ell \ge 2$, or $\alpha_\ell=1$ and either $\leg{\Delta}{\ell} \ne 0$ or $\leg{\Delta_K}{\ell} =0$. Let $\Ss_2$ be the complementary set of exact prime powers divisors of $n$. Of course, $\Ss_2$ actually consists only of primes.  Referring back to \eqref{eq:bigdiv},
\begin{equation}\label{eq:4}
 \prod_{\ell^{\alpha_{\ell}} \in \Ss_1} \lambda_{\ell^{\alpha_{\ell}}} \prod_{\ell \in \Ss_2} \lambda_{\ell} \mid w_K \frac{d}{2h_K} \prod_{\ell \in \Ss_2}\ell.
  \end{equation}

On the other hand, Theorem \ref{BIGSPY} implies
\[ \lambda_{\ell^{\alpha_{\ell}}} \mid w_K \cdot [F\cap K^{(\ell^{b_{\ell}} \O_K)}: K^{(\O_K)}] \]
for each prime $\ell$ dividing $n$. The fields $F\cap K^{(\ell^{b_{\ell}} \O_K)}$ are linearly disjoint extensions of $K^{(\O_K)}$, all contained in $F$. Thus, with $m\coloneqq \omega(n)$,
\begin{equation}\label{eq:5} \prod_{\ell^{\alpha_{\ell}} \in \Ss_1} \lambda_{\ell^{\alpha_{\ell}}} \prod_{\ell \in \Ss_2} \lambda_{\ell} \mid w_K^{m} \cdot [F: K^{(\O_K)}] = w_K^{m}\cdot \frac{d}{2h_K}.
\end{equation}
Putting \eqref{eq:4} and \eqref{eq:5} together, we find
\begin{equation}\label{eq:6}
 \prod_{\ell^{\alpha_{\ell}} \in \Ss_1} \lambda_{\ell^{\alpha_{\ell}}} \prod_{\ell \in \Ss_2} \lambda_{\ell} \mid w_K \frac{d}{2h_K} \prod_{\ell \in \Ss_2,  \, \ell \mid w_K}\ell.
  \end{equation}
If $w_K=2$, it follows that
\[
\prod_{\ell^{\alpha_{\ell}} \in \Ss_1} \lambda_{\ell^{\alpha_{\ell}}} \prod_{\ell \in \Ss_2} \lambda_{\ell} \mid 2 \frac{d}{h_K}.
\]
In fact, if $w_K=4$, the same divisibility condition holds. Indeed, 2 is the only prime that divides $w_K$, but $2 \notin \Ss_2$ since 2 ramifies in $K=\Q(i)$. If $w_K=6$, then $3 \notin \Ss_2$ since 3 ramifies in $K=\Q(\sqrt{-3})$, and  \eqref{eq:6} implies
\[
\prod_{\ell^{\alpha_{\ell}} \in \Ss_1} \lambda_{\ell^{\alpha_{\ell}}} \prod_{\ell \in \Ss_2} \lambda_{\ell} \mid 6 \frac{d}{h_K}.  \qedhere \]
\end{proof}
\noindent
As a consequence, in the case of $\O_K$-CM elliptic curves, we recover the SPY
Bounds as divisibilities.

\begin{cor}[SPY Divisibilities]\label{cor:SPYdiv}
Let $F$ be a number field of degree $d$ containing an imaginary quadratic field $K$,
and let $E_{/F}$ be an $\O_K$-CM elliptic curve.  If $E$ has an $F$-rational point of order $N$, then
\[ h_K \varphi(N) \mid \frac{w_K}{2} \cdot d. \]
\end{cor}
\begin{proof}
Suppose $E_{/F}$ has a point of order $N=\prod \ell^{e_{\ell}}$. For each $\ell \mid N$,
\[
E(F)[\ell^{\infty}] \cong \Z/\ell^{a_{\ell}}\Z\times \Z/\ell^{b_{\ell}}\Z,
\]
where $b_{\ell} \ge a_{\ell} \ge 0$ and $b_{\ell} \ge e_{\ell}$. Since $E$ has CM by the maximal order, there are no primes of type $\Ss_2$, and for each $\ell^{\alpha_{\ell}} \in \Ss_1$ we have $\varphi(\ell^{b_{\ell}}) \mid \lambda_{\ell^{\alpha_{\ell}}}$. Thus by (\ref{eq:4}) we have
 \[
 \varphi(N) = \prod_{\ell \mid N} \varphi(\ell^{e_{\ell}}) \mid \prod_{\ell \mid N} \varphi(\ell^{b_{\ell}}) \mid \prod_{\ell^{\alpha_{\ell}} \in \Ss_1} \lambda_{\ell^{\alpha_{\ell}}} \mid  w_K \frac{d}{2h_K}.
 \qedhere
 \] \end{proof}

 \begin{remarks} Let us discuss the sharpness of the divisibilities obtained in Theorem \ref{BIGSPY}.
\begin{enumerate}
\item[(a)] If $\ell \neq 2$ and $a = b$, then in every case Theorem \ref{BIGSPY}
gives
\[ \ell^{2b-2}(\ell-1)(\ell-\leg{\Delta_K}{\ell})/w_K \mid [K^{(\ell^b\O_K)}: K^{(\O_K)}]. \]
Since in fact we have
\[  [K^{(\ell^b\O_K)}: K^{(\O_K)}] = \ell^{2b-2}(\ell-1)(\ell-\leg{\Delta_K}{\ell})/w_K,  \]
Theorem \ref{BIGSPY} is sharp in this case, which includes all of Case (i).

\item[(b)] If $\leg{\Delta}{\ell}=1$ and $a=0$, the image of the $\ell$-adic Galois representation lands in a split Cartan subgroup (cf. \cite[$\S 3.4$]{BCS15}).  Thus for all $n \in \Z^+$ we have an $F$-rational subgroup of order $\ell^n$. If $\ell$ is an odd prime, it follows from \cite[Theorem 7.2]{BCS15} that there is an $\O_K$-CM elliptic curve $E$ defined over an extension $L/K^{(\O_K)}$ with $[L:K^{(\O_K)}]=\varphi(\ell^n)/2$ such that $E(L)$ contains a point of order $\ell^n$. Thus the divisibility condition given is best possible when $w_K=2$ and $\ell$ is odd.

\item[(c)] In Theorem \ref{BIGSPY} we recorded the divisibilities in terms of
$[F \cap K^{(\ell^b \O_K)}:K^{(\O_K)}]$ rather than in terms of $[F:K] =
[F:K^{(\O_K)}] h_K$ because we get a stronger result by doing so.  However,
it may be more natural to ask for best possible
divisibilities of $[F:K]$.  In part (b) above, the optimality occurs in
this stronger sense.  As for part (a), when $\ell$ does not divide the
conductor $\ff$ of the order $\O$, classical CM theory implies that there is an
elliptic curve defined over $K^{(\ell^b)}$ with full $\ell^b$-torsion
and thus multiplying the bound of Theorem \ref{BIGSPY} by $h_K$ gives
the optimal divisibility of $[F:K]$ in this case.
\item[(d)] The field $F$ also contains the ring class field $K(\O)$
of the order $\O$.  Let $\ff_{\ell} = {\ord_{\ell}(\mathfrak{f}(\O))}$ and suppose
that $\ff_{\ell} \geq 1$.  (This is the condition under which we cannot reduce to the case of $\O_K$-CM.)   For all $\ell > 2$ we have
\[\ord_{\ell} [K(\O):K^{(\O_K)}] = \ell^{\ff_{\ell} - 1}, \]
so if $\ff_{\ell} > 2b-|\leg{\Delta_K}{\ell}|$ then there is a larger power of $\ell$ dividing
$[F:K^{(\O_K)}]$ than is given by Theorem \ref{BIGSPY}.  (This does not say
that Theorem \ref{BIGSPY} is not optimal but rather that it could be refined
by considering an additional parameter.)
 \item[(e)] In case (v) of Theorem \ref{BIGSPY}, there are values of $a$ and $b$ for which we suspect that the divisibility on $d = [F:K]$, at least, can be improved.  Suppose $w_K=2$, $b=2$, $a=0$ and  $\leg{\Delta_K}{\ell}=1$.  In this case Theorem \ref{BIGSPY} implies $h_K (\ell-1)^2 \mid d$, whereas the SPY bounds here give $\ell(\ell-1) \leq d$: this is not quite implied by our result!  In light of Corollary \ref{cor:SPYdiv} it is reasonable to expect in all cases the SPY bounds may be multiplied
by a factor of $h_K$ and yield divisibilities.\footnote{In fact, we believe that Silverberg's arguments can be easily adapted to yield these strengthenings.  We will revisit this in a later work.}  If so, the two results would combine to give $h_K \ell(\ell-1)^2 \mid d$.  Note that by part (d) this certainly occurs if
$\ff_{\ell} \geq 2$, so the open case is precisely $\ff_{\ell} = 1$.
\end{enumerate}
\end{remarks}

\section{Proof of Theorem \ref{thm:density}: Typical boundedness of $T_{\rm CM}(d)$}
\noindent
We need a result from the part of number theory known as the `anatomy of integers'.

\begin{prop}[Erd\H{o}s--Wagstaff {\cite[Theorem 2]{EW80}}]\label{lem:EW}  For all
$\epsilon > 0$, there is a positive integer $B'_{\epsilon}$ such that the
set of positive integers which are divisible by $\ell-1$ for some
prime $\ell > B'_{\epsilon}$ has upper density at most $\epsilon$.
\end{prop}

\begin{proof}[Proof of Theorem \ref{thm:density}] Suppose that
\begin{equation}\label{eq:TCMbig} T_{\mathrm{CM}}(d) > B.\end{equation} We will see that if $B$ is a constant chosen sufficiently large in terms of $\epsilon$, then for large $x$ the inequality \eqref{eq:TCMbig} has fewer than $\epsilon x$ solutions $d\le x$.

Choose a degree $d$ number field $F$ and a CM elliptic curve $E_{/F}$ with $\#E(F)[{\rm tors}] > B$. Let $K$ denote the CM field. Suppose to start with that $\#E(F)[{\rm tors}]$ has a prime factor $\ell > B'+1$, where $B' = B'_{\epsilon/24}$, in the notation of Proposition \ref{lem:EW}. Since $\ell$ divides $\#E'(FK)[{\rm tors}]$,  Lemma \ref{lem:div} shows that
\[ \ell-1 \mid w_K \frac{[FK:\Q]}{2} \mid w_K d \mid 12d. \]
Note that $12d \le 12x$. By the definition of $B'$, once $x$ is large, there are fewer than $\frac{\epsilon}{24} \cdot 12x = \frac{\epsilon}{2} x$ possibilities for $12d$, and so also at most $\frac{\epsilon}{2} x$ possibilities for $d$.

Now suppose instead that each prime factor of $\#E(F)[{\rm tors}]$ is at most $B'+1$. Then $\#E(F)[{\rm tors}]$ has at most $r\coloneqq \pi(B'+1)$ distinct prime factors, and so we can choose a prime power $\ell^{\alpha} \parallel \#E(F)[{\rm tors}]$ with
\[ \ell^{\alpha} \ge (\#E(F)[{\rm tors}])^{1/r} > B^{1/r}. \]
Let us impose the restriction that $B \ge (B'+1)^{r}$. Then $\ell^{\alpha} > B'+1 \ge \ell$, and so $\alpha \ge 2$. Applying Lemma \ref{lem:div} in the same manner as above, we find that $12d$ is divisible by either $\ell^{\alpha-2}(\ell^2-1)$, $\ell^{\alpha-1} (\ell-1)$, or $\ell^{\alpha-2}(\ell-1)^2$. Thus, the number of possibilities for $12d$ is bounded by
\begin{align*} 12x\left(\frac{1}{\ell^{\alpha-2}(\ell^2-1)} + \frac{1}{\ell^{\alpha-1}(\ell-1)} + \frac{1}{\ell^{\alpha-2}(\ell-1)^2}\right) &\le 12x\left(\frac{4/3}{\ell^{\alpha}}  + \frac{2}{\ell^{\alpha}} + \frac{4}{\ell^{\alpha}}\right) \\&< 100 \frac{x}{\ell^{\alpha}}. \end{align*}
Now sum on the possible values of $\ell^{\alpha}$. We find that the number of choices for $d$ is at most
\[ 100x\sum_{\substack{\ell^{\alpha} > B^{1/r} \\ \ell \le B'+1 \\ \alpha \ge 2}} \frac{1}{\ell^{\alpha}} = 100x \sum_{\ell \le B'+1} \sum_{\substack{\alpha:\,\alpha \ge 2 \\ \ell^{\alpha} > B^{1/r}}} \frac{1}{\ell^{\alpha}}. \]
The geometric series appearing as the inner sum is at most twice its largest term; this yields an upper bound for the right-hand side of $\frac{200r}{B^{1/r}} x$.
Now impose the additional restriction that $B> (\frac{400r}{\epsilon})^{r}$. Then our upper bound here is smaller than $\frac{\epsilon}{2}x$. Putting this together with the result of the last paragraph finishes the proof.
\end{proof}

\begin{rmk} By a more refined analysis, using techniques recently developed to study the range of Carmichael's $\lambda$-function \cite{LP14, FLP14}, one can establish the following sharpening of Theorem \ref{thm:density}: as $B\to\infty$, the upper and lower densities of $\{n\mid T_{\rm CM}(d) > B\}$ both take the form $(\log{B})^{-\eta+o(1)}$. Here
\[ \eta=1-\frac{1+\log\log{2}}{\log{2}}=0.08607\dots,\] the \emph{Erd\H{o}s--Ford--Tenenbaum constant}. Details will be presented elsewhere.
\end{rmk}

\section{Proof of Theorem \ref{OLSONDEGREETHM}: Characterization of Olson degrees}
\noindent
As already mentioned in the introduction, any group that appears as the torsion subgroup of a CM elliptic curve over a degree $d$ number field also appears over some degree $d'$ number field, for each multiple $d'$ of $d$ (see \cite[Theorem 2.1(a)]{BCS15}). So the set of non-Olson degrees is indeed a set of multiples.

To prove that the set $\Gg$ appearing in the statement of Theorem \ref{OLSONDEGREETHM} is a set of generators, we need the following results from \cite{BCS15}.

\begin{prop}[{\cite[Theorem 4.9]{BCS15}}]\label{prop:realcyclo} Let $F$ be a number field that admits a real embedding, and let $E_{/F}$ be a $K$-CM elliptic curve. If $E(F)$ contains a point of order $n$, then $\Q(\zeta_n) \subset FK$.
\end{prop}

\begin{prop}[{\cite[Theorem 7.1]{BCS15}}]\label{prop:oddthm} Let $F$ be a number field of odd degree, and let $E_{/F}$ be a CM elliptic curve. Then $E(F)[{\rm tors}]$ is isomorphic to one of the following groups:
\begin{enumerate}
\item the trivial group $\{\bullet\}$, $\Z/2\Z$, $\Z/4\Z$, or $\Z/2\Z \times \Z/2\Z$,
\item the group $\Z/\ell^n \Z$ for a prime $\ell \equiv 3\pmod{8}$ and some positive integer $n$,
\item the group $\Z/2\ell^n \Z$ for a prime $\ell \equiv 3\pmod{4}$ and some positive integer $n$.
\end{enumerate}
Conversely, each of these groups appears as the torsion subgroup of some CM elliptic curve over some odd degree number field.
\end{prop}

\begin{prop}[{\cite[Corollary 7.5]{BCS15}}]\label{prop:nonolson} Let $\O$ be an imaginary quadratic order of discriminant $\Delta$, and let $\ell >2$ be a prime dividing $\Delta$. There is a number field $L$ of degree $\frac{\ell-1}{2} \cdot h(\O)$ and an $\O$-CM elliptic curve $E_{/L}$ with an $L$-rational point of order $\ell$.
\end{prop}

\begin{proof}[Proof of Theorem \ref{OLSONDEGREETHM}]
First we verify that any $d \in \Gg$ is non-Olson.  By \cite[Theorem 1.4]{BCS15},  $2$ is a non-Olson degree.  It remains to consider $d=\frac{\ell-1}{2} \cdot h_{\Q(\sqrt{-\ell})}$ for a prime $\ell>3$ with $\ell \equiv 3 \pmod{4}$. Let  $K=\Q(\sqrt{-\ell})$. By Proposition \ref{prop:nonolson}, there is an $\Ok$-CM elliptic curve $E$ defined over a number field $L$ of degree $\frac{\ell-1}{2} \cdot h_{\Q(\sqrt{-\ell})}$ such that $E(L)$ contains a point of order $\ell$. Thus $E(L)[\tors]$ is not an Olson group and $d$ is a non-Olson degree.

Next, we suppose $d$ is a non-Olson degree and show $d\in M(\Gg)$. There is an elliptic curve $E$ defined over a number field $F$ of degree $d$ for which $E(F)[\tors]$ is not an Olson group. Since $2 \in \Gg$, we may assume that $d$ is odd and hence that $F$ admits a real embedding.

By Proposition \ref{prop:oddthm}, $E(F)$ contains a point of prime order $\ell$ where $\ell \equiv 3 \pmod{4}$. By Proposition \ref{prop:realcyclo}, $\Q(\zeta_\ell) \subset FK$, where $K$ is the CM field. Thus, $FK$ contains the quadratic subfield $\Q(\sqrt{-\ell})$ of $\Q(\zeta_\ell)$.  Since $4\nmid [FK:\Q]$, the field $FK$ can contain only one quadratic subfield, and so $K= \Q(\sqrt{-\ell})$.

Suppose first that $\ell > 3$. Then Lemma \ref{lem:div} shows that $h_K \cdot (\ell-1) \mid w_K \frac{[FK:\Q]}{2} = 2 d$. Thus $h_K \cdot \frac{\ell-1}{2}\mid d$ and $d \in M(\Gg)$. Now suppose $\ell = 3$.  Since $E(F)[{\rm tors}]$ is not Olson, it must have a point of order $9$.  By Proposition \ref{prop:realcyclo}, $\Q(\zeta_{9}) \subset FK$. Thus $6 \mid [FK:\Q]=2d$, so $3\mid d$. But $3=\frac{7-1}{2} \cdot h_{\Q(\sqrt{-7})}$, so again $d \in M(\Gg)$.
\end{proof}

\section{Proof of Theorem \ref{thm:density2}: Olson degrees have positive density}
\noindent
Theorem \ref{thm:density2} follows from Theorem \ref{OLSONDEGREETHM} together with the following  elementary result from the theory of sets of multiples.

\begin{lem}\label{lem:hall} Let $\Gg \subset \Z^+$. If $\sum_{g \in \Gg}\frac{1}{g} < \infty$, then $M(\Gg)$ has an asymptotic density. If moreover $1\notin\Gg$, then the density of $M(\Gg)$ is strictly less than $1$.
\end{lem}
\begin{proof}
See Theorem 0.1 and Corollary 0.10 in Chapter 0 of Hall's monograph \cite{hall96}.
\end{proof}
We can now prove Theorem \ref{thm:density2}.

\begin{proof}[Proof of Theorem \ref{thm:density2}] In view of Lemma \ref{lem:hall}, it suffices to show that $\sum_{g \in \Gg}\frac{1}{g}<\infty$, where $\Gg$ is the set defined in Theorem \ref{OLSONDEGREETHM}. Siegel's theorem (see for instance \cite[p. 124]{IK04}) implies that for each $\epsilon > 0$,
\[ \frac{\ell-1}{2} \cdot h_{\Q(\sqrt{-\ell})} \gg_{\epsilon} \ell^{3/2-\epsilon}. \]
Fixing any $\epsilon < \frac{1}{2}$, we obtain the desired convergence. Alternatively, the work of Goldfeld--Gross--Zagier yields an effective lower bound $\frac{\ell-1}{2} \cdot h_{\Q(\sqrt{-\ell})} \gg_{\epsilon} \ell (\log{\ell})^{1-\epsilon}$ (see \cite[p. 540]{IK04}). Now fixing $\epsilon \in (0,1)$, partial summation along with the prime number theorem gives that $\sum_{\ell}\frac{1}{\ell (\log{\ell})^{1-\epsilon}} < \infty$. 	
\end{proof}

\begin{rmk} By another appeal to Proposition \ref{lem:EW}, one can prove Theorem \ref{thm:density2} without using any lower bounds on $h_{\Q(\sqrt{-\ell})}$. Compare with the proof of \cite[Theorem 4]{PS88}.
\end{rmk}

\section{Proof of Theorem \ref{thm:primepower}: Prime power Olson degrees}

\begin{proof}[Proof of Theorem \ref{thm:primepower}]If $p\le 5$, then $p$ and its powers are non-Olson degrees, so we assume that $p \ge 7$. Suppose that $p^n$ is not an Olson degree. From the classification of Olson degrees (Theorem \ref{OLSONDEGREETHM}), there is a prime $\ell > 3$ with $\ell \equiv 3\pmod{4}$ for which $\frac{\ell-1}{2} \cdot h_{\Q(\sqrt{-\ell})} \mid p^n$. Hence, there are integers  $r \ge 1$ and $s\ge 0$ with $r+s \le n$,
\[ \frac{\ell-1}{2} = p^r,\quad\text{and}\quad h_{\Q(\sqrt{-\ell})} = p^s. \]
We argue that $p$ is  bounded (ineffectively) in terms of $n$.
By Siegel's theorem, if $p$ is large in terms of $n$, then $h_{\Q(\sqrt{-\ell})} > \ell^{\frac{1}{2}-\frac{1}{3n}} > p^{\frac{r}{2}-\frac{1}{3}}$.
Using the elementary explicit upper bound
\begin{equation}\label{eq:hboundLP} h_{\Q(\sqrt{-\ell})} \le \ell^{1/2} \log\ell, \end{equation}
(see, e.g., \cite[\S2]{LP92}) we find that for $p$ large enough in terms of $n$, we also have $h_{\Q(\sqrt{-\ell})} < p^{\frac{r}{2}+\frac{1}{3}}$. Thus, $p^{-1/3} < p^{s-\frac{r}{2}} < p^{1/3}$. Since $s-r/2$ is an integer or half-integer, we must have $s=r/2$. In particular, $r=2s$ is even. But then $\ell = 2p^{2s}+1 \equiv 0\pmod{3}$, contradicting that $\ell > 3$.\end{proof}

\begin{rmk} For general $n$, the ineffectivity of Siegel's theorem prevents us from giving a
concrete bound on the largest non-Olson prime power $p^n$. However, as we explain below, the above argument can be made effective when $n=1,2$, or $3$. In this way, we obtain a simple proof that $p^n$ is Olson for every $p > 5$. (Recall that when $n=1$, this was proved already in \cite{BCS15}.)

Given a counterexample, choose $\ell, r$, and $s$ as in the above proof. As before, working modulo $3$ shows that $r$ is odd. To finish the proof, it suffices to prove that $s=0$, i.e., $h_{\Q(\sqrt{-\ell})} = 1$.
To see that this is enough, notice that $\ell = 2p+1$ or $2p^3+1$, where $p > 5$, so that $\ell > 11$. Now if $K$ is an imaginary quadratic field with $h_K=1$, an elementary argument shows that every prime smaller than $\frac{1+|\Delta_K|}{4}$ is inert in $K$. In particular, $3$ is inert in $\Q(\sqrt{-\ell})$, forcing $3 \mid \ell-1$ and thus $3\mid p$. But this contradicts that $p > 5$.

Now we prove that $s=0$. If $r=3$, the inequality $r+s \le 3$ immediately forces $s=0$. If $r=1$, so that $\ell=2p+1$, then \eqref{eq:hboundLP} implies that $s=0$ for all $p\ge 41$. For $5 < p < 41$, we check directly that there is no case where $\ell=2p+1$ is prime and $h_{\Q(\sqrt{-\ell})}$ is a power of $p$.
\end{rmk}

\section{Proof of Theorem \ref{thm:avg0}: Averages of $T_{\rm CM}(d)$}\label{sec:average}
\subsection{The average over odd $d$}\noindent Since the results for odd $d$ are easier to obtain, we start there.

\begin{proof}[Proof of the upper bound in Theorem \ref{thm:avg0}(ii)] Recall
that $T_{\CM}(d) \geq 6$ for all positive integers $d$.  Thus, from Proposition \ref{prop:oddthm}, we may assume that $T_{\rm CM}(d) = \ell^\alpha$ or $2\ell^\alpha$ for some prime $\ell \equiv 3\pmod{4}$ and some positive integer $\alpha$.

For any curve achieving the maximum indicated by $T_{{\rm CM}}(d)$, the CM field must be $\Q(\sqrt{-\ell})$, for the same reason as in the proof of Theorem \ref{OLSONDEGREETHM}. Now we apply Lemma \ref{lem:div} to bound the number of possible values of $d\le x$, given that $\ell^{\alpha}$ divides $\#E(F)[{\rm tors}]$.
By a calculation similar to that seen in the proof of Theorem \ref{thm:density}, the number of such $d$ is at most $100\frac{x}{h_{\Q(\sqrt{-\ell})} \cdot\ell^{\alpha}}$. So given $\ell^\alpha$, the contribution to $\sum_{d \le x,~2\nmid d} T_{{\rm CM}}(d)$ from these $d$ is at most $100\frac{x}{h_{\Q(\sqrt{-\ell})} \cdot\ell^{\alpha}} \cdot 2\ell^{\alpha}=200 x/h_{\Q(\sqrt{-\ell})}$.

We now sum on the possibilities for $\ell^{\alpha}$. Since $\ell^\alpha\le 100x$, there are $O(\log{x})$ possible values of $\alpha$.  Moreover, the only values of $\ell$ that can occur are those with $\ell \cdot h_{\Q(\sqrt{-\ell})} \le 100 x$. Fix a small $\epsilon > 0$. Recalling Siegel's lower bound $h_{\Q(\sqrt{-\ell})} \gg \ell^{1/2-\epsilon}$, we find that $\ell \le x^{2/3+\epsilon}$ (assuming $x$ is sufficiently large). Hence,
\[ \sum_{\ell^{\alpha}} 200 \frac{x}{h_{\Q(\sqrt{-\ell})}} \ll x\log{x}  \sum_{\ell \le x^{2/3+\epsilon}} \frac{1}{\ell^{1/2-\epsilon}} \ll x \log{x} \cdot (x^{2/3+\epsilon})^{1/2+\epsilon} \ll x^{4/3+2\epsilon}.  \]
Since $\epsilon$ may be taken arbitrarily small, the upper bound follows.\end{proof}

\begin{proof}[Proof of the lower bound in Theorem \ref{thm:avg0}(ii)] Here the main difficulty is the need to avoid double counting.

Fix a small $\epsilon > 0$. For large $x$, let $Y = x^{2/3-\epsilon}$,  and let $\Pp_0$ be the set of primes $\ell \equiv 3\pmod{4}$ belonging to $[Y,2Y]$. Then $\#\Pp_0\gg Y/\log Y$. We prune the set $\Pp_0$ as follows. Let $\ell_1$ be any element of $\Pp_0$. Remove from $\Pp_0$ all $\ell$ for which $\frac{\ell-1}{2} \mid \frac{\ell_1-1}{2} \cdot h_{\Q(\sqrt{-\ell_1})}$. Now let $\ell_2$ be any remaining element, and remove all $\ell$ for which $\frac{\ell-1}{2} \mid \frac{\ell_2-1}{2} \cdot h_{\Q(\sqrt{-\ell_2})}$. We continue in the same way until all elements of $\Pp_0$ are exhausted. Let $\Pp$ be the set $\ell_1, \ell_2, \ell_3, \dots$. The maximal order of the divisor function (see \cite[Theorem 315, p. 343]{HW08}) shows that the number of primes removed at each step in the construction of $\Pp$ is smaller than $x^{\epsilon/2}$, and so $\#\Pp \ge x^{2/3-2\epsilon}$.

By construction, as $\ell$ ranges over $\Pp$, the products $\frac{\ell-1}{2} \cdot h_{\Q(\sqrt{-\ell})}$
are all distinct. By genus theory, all of these products are odd. Since $\ell \le 2Y$ and $h_{\Q(\sqrt{-\ell})} \le \ell^{1/2}\log{\ell}$, we find that each $\frac{\ell-1}{2} \cdot h_{\Q(\sqrt{-\ell})} \le x$. Putting all of this together with Proposition \ref{prop:nonolson},
\[ \sum_{\substack{d \le x \\ 2\nmid d}} T_{\rm CM}(d) \ge \sum_{\ell \in \Pp} T_{\rm CM}\bigg(\frac{\ell-1}{2} \cdot h_{\Q(\sqrt{-\ell})}\bigg) \ge \sum_{\ell \in \Pp} \ell \ge Y \cdot \#\Pp \ge x^{4/3-3\epsilon}. \]
Since $\epsilon$ can be taken arbitrarily small, we obtain the lower bound.
\end{proof}

\subsection{The unrestricted average} We will use the following result.

\begin{prop}[{\cite[Theorem 1(a)]{CCS13}}]\label{prop:CCS1A} For every prime $\ell\equiv 1\pmod{3}$, there is an elliptic curve $E$ with $j(E)=0$ over a number field $F$ of degree $\frac{\ell-1}{3}$, with $E(F)$ containing a point of order $\ell$.\end{prop}

\begin{proof}[Proof of the lower bound in Theorem \ref{thm:avg0}(i)]  Immediately from Proposition \ref{prop:CCS1A},
\[ \sum_{d \le x}T_{\rm CM}(d) \ge \sum_{\substack{x < \ell \le 3x \\\ell \equiv 1\pmod{3}}} T_{\rm CM}\bigg(\frac{\ell-1}{3}\bigg) \ge \sum_{\substack{x < \ell \le 3x \\\ell \equiv 1\pmod{3}}}\ell \ge x \sum_{\substack{x < \ell \le 3x \\ \ell \equiv 1\pmod{3}}} 1 \gg \frac{x^2}{\log{x}}. \qedhere \]
\end{proof}
\noindent
The proof of the upper bound is considerably more intricate. The needed methods are similar to those used by Erd\H{o}s to estimate the counting function of the range of the Euler $\varphi$-function \cite{erdos35}. To continue, we need two further `anatomical' results.

\begin{lem}\label{lem:HRHT}\mbox{ }
\begin{enumerate}\item  There are positive numbers $C_1$ and $C_2$ such
that for all $k \in \Z^+$ and all real numbers $x \geq 3$, we have
\[ \# \{ d \leq x \mid \omega(d) = k\} \leq C_1 \frac{x}{\log x} \frac{(\log \log x +
C_2)^{k-1}}{(k-1)!}. \]
\item There is a positive number $C_3$ such that for
all $K \in \Z^+$ and all real numbers $x \geq 3$, we have
\[ \# \{d \leq x \mid \Omega(n) \geq K\} \leq C_3 \frac{K}{2^K} x\log{x}. \]
\end{enumerate}
\end{lem}
\begin{proof}
Part (i) is a classical inequality of Hardy and Ramanujan \cite{HR00}. Part (ii) is taken from \cite{HT88} (Exercise 05, p. 12); for details, see the proofs of Lemmas 12 and 13 in \cite{LP07}.
\end{proof}
To prove the upper bound in Theorem \ref{thm:avg0}(i), we will show that the mass of $T_{\rm CM}(d)$ is highly concentrated on certain arithmetically special $d$.

For each positive integer $n$, we form a set of integers $\Lambda(n)$, with definition motivated by the statement of Theorem \ref{lem:div2}. For each prime power $\ell^{\alpha}$ with $\alpha \ge 2$, let
\[ \Lambda(\ell^{\alpha}) = \{\ell^{\alpha-2}(\ell-1)(\ell+1), \ell^{\alpha-2}(\ell-1)^2, \ell^{\alpha-1}(\ell-1)\}, \]
and for each prime $\ell$, let
\[ \Lambda(\ell) = \{\ell^2-1, (\ell-1)^2, \ell-1\}. \]
For any $n \in \Z^+$, let $\Lambda(n)$ be the set of integers $\lambda$ that can be written in the form
\begin{equation}\label{eq:lambdadef} \prod_{\ell^{\alpha} \parallel n} \lambda_{\ell^{\alpha}},  \end{equation}
where each $\lambda_{\ell^{\alpha}} \in \Lambda(\ell^{\alpha})$.

\begin{lem}\label{lem:lambdadescription} Let $n$ be a positive integer.
\begin{enumerate}
\item[(i)] The cardinality of $\Lambda(n)$ is bounded above by $3^{\omega(n)}$.
\item[(ii)] Each $\lambda \in \Lambda(n)$ satisfies
\[ \lambda \gg n/(\log\log{(3n)})^2, \]
where the implied constant is absolute.
\item[(iii)] Each $\lambda \in \Lambda(n)$ has
\[ \Omega(\lambda) \ge \Omega(n)-2. \]
\end{enumerate}
\end{lem}
\begin{proof} Since $\#\Lambda(\ell^{\alpha}) = 3$ for each prime power $\ell^{\alpha}$, (i) is immediate. To prove (ii), notice that each $\lambda_{\ell^{\alpha}} \in \Lambda(\ell^{\alpha})$ satisfies $\lambda_{\ell^{\alpha}} \ge \ell^{\alpha}(1-1/\ell)^2$. Consequently, each $\lambda \in \Lambda(n)$ is bounded below by $n \prod_{\ell \mid n}(1-1/\ell)^2 = \varphi(n)^2/n$. The claim now follows from the estimate $\varphi(n) \gg n/\log\log(3n)$ (see, e.g., \cite[Theorem 323, p. 352]{HW08}). For (iii), observe that except in the case $\ell=2$, each $\lambda_{\ell^{\alpha}} \in \Lambda_{\ell^{\alpha}}$ has $\Omega(\lambda_{\ell^{\alpha}}) \ge \alpha$, and that when $\ell=2$, we have the weaker bound  $\Omega(\lambda_{\ell^{\alpha}}) \ge \alpha-2$.
\end{proof}

\begin{proof}[Proof of the upper bound in Theorem \ref{thm:avg0}(i)] For even $d$, let $T'_{{\rm CM}}(d)$ be defined in the same way as $T_{\rm CM}(d)$, but with the extra restriction that $E$ is defined over a degree $d$ number field $F$ \emph{containing the CM field of $E$}. Since we can replace $F$ by a quadratic extension $F'/F$ containing the CM field, we have $T_{\rm CM}(d) \le T'_{\rm CM}(2d)$ for all $d$. Thus, it suffices to establish the claimed upper bound for $\sum_{d  \le x} T'_{\rm CM}(2d)$. The contribution to this latter sum from values of $d$ with $T'_{\rm CM}(2d) \le x/\log{x}$ is trivially $O(x^2/\log{x})$, which is acceptable for us. Since $T'_{\rm CM}(2d) \le T_{\rm CM}(2d) \le C x\log\log{x}$ for a certain absolute constant $C$ (see Theorem 1 of \cite{CP15}), the contribution from the remaining values of $d$ is
\[ \ll x\log\log{x} \sum_{\substack{d \le x \\ T'_{\rm CM}(2d) > \frac{x}{\log{x}}}} 1. \]
The proof of the theorem will be completed if we show that
\begin{equation}\label{eq:TCM2dlarge} \sum_{\substack{d \le x \\ T'_{\rm CM}(2d) > \frac{x}{\log{x}}}} 1 \le \frac{x}{(\log{x})^{1+o(1)}}, \end{equation}
as $x\to\infty$. To this end, suppose $T'_{\rm CM}(2d) = n > x/\log{x}$. From Theorem \ref{lem:div2}, $12d$ is divisible by some $\lambda \in \Lambda(n)$. So with
\[ \Lambda' \coloneqq  \bigcup_{\frac{x}{\log{x}} < n \le Cx\log\log{x}} \Lambda(n), \]
we see that
\begin{equation}\label{eq:ddivvy} \sum_{\substack{d \le x \\ T'_{\rm CM}(2d) > \frac{x}{\log{x}}}} 1 \le \#\{D \le 12x: \lambda \mid D\text{ for some }\lambda \in \Lambda'\}.\end{equation}
We bound the right-hand side of \eqref{eq:ddivvy} from above by considering various (possibly overlapping) cases for $\lambda$. For notational convenience, we put $X = Cx\log\log{x}$. We let $\epsilon > 0$ be a small, fixed parameter.

\vskip 0.1in
\noindent \textbf{Case I: } $\lambda \in \Lambda(n)$ for an $n \in (\frac{x}{\log{x}},X]$ with $\omega(n) \le \eta \log\log{x}$, where $\eta > 0$ is a sufficiently small constant. ``Sufficiently small'' is allowed to depend on $\epsilon$, and will be specified in the course of the proof.

Using the lower bound from Lemma \ref{lem:lambdadescription} on the elements of $\Lambda(n)$, we see that the number of $D \le 12x$ divisible by some $\lambda \in \Lambda(n)$ is
\[ \ll x\sum_{\lambda \in \Lambda(n)}\frac{1}{\lambda}\ll \frac{x}{n}(\log\log{x})^2 \sum_{\lambda \in \Lambda(n)} 1  \ll \frac{x}{n}(\log\log{x})^2 \cdot 3^{\omega(n)} \ll \frac{x}{n}(\log\log{x})^2 (\log{x})^{\eta\log{3}} .
 \]
If we assume that $\eta < \epsilon/\log{3}$, this upper bound is $O(\frac{x}{n} (\log{x})^{2\epsilon})$. Thus, the total number of $D$ that can arise in this way is
\begin{equation}\label{eq:totnumber} \ll x (\log{x})^{2\epsilon} \sum_{\substack{\frac{x}{\log{x}} < n \le X \\ \omega(n) \le \eta\log\log{x}}}{\frac{1}{n}}. \end{equation}
To estimate the sum we appeal to Lemma \ref{lem:HRHT}(i).  For each $T \in [x/\log{x},X]$, the number of $n \le 2T$ with $\omega(n) \le \eta\log\log{x}$ is \[\ll \frac{T}{\log{x}} \sum_{1\le k \le \eta\log\log{x}} \frac{(\log\log{x}+O(1))^{k-1}}{(k-1)!}.\]
We can assume $\eta < \frac{1}{2}$. Then each term in the right-hand sum on $k$ is at most half of its successor (once $x$ is large). Hence, the sum is bounded by twice its final term. Recalling that $(k-1)! \ge ((k-1)/e)^{k-1}$, the expression in the preceding display is thus seen to be $O(T (\log{x})^{\eta \log(e/\eta)-1+\epsilon})$.
Hence,
\begin{align*} \sum_{\substack{n  \in [T,2T] \\ \omega(n) \le \eta\log\log{x}}}{\frac{1}{n}} &\le \frac{1}{T} \#\{n \le 2T\mid \omega(n) \le \eta \log\log{x}\} \\&\ll(\log{x})^{\eta \log(e/\eta)-1+\epsilon}.  \end{align*}
Letting $T$ range over the $O(\log\log{x})$ values of the form $T=2^j x/\log{x}$, where $j \ge 0$ and  $2^j x/\log{x} \le X$, we find that
\[ \sum_{\substack{\frac{x}{\log{x}} < n \le X \\ \omega(n) \le \eta\log\log{x}}}{\frac{1}{n}}\ll(\log{x})^{\eta \log(e/\eta)-1+2\epsilon}. \]
Substituting this into \eqref{eq:totnumber}, and choosing $\eta$ sufficiently small in terms of $\epsilon$, we get that the total number of $D$ arising in this case is $O(x (\log{x})^{5\epsilon} (\log{x})^{-1})$.
\vskip 0.1in
\noindent \textbf{Case II: } $\lambda \in \Lambda(n)$ for an $n \in (\frac{x}{\log{x}},X]$ with $\eta \log\log{x} < \omega(n) < 10 \log\log{x}$ and
\[ \sum_{\substack{\ell \mid n \\ \Omega(\ell-1) \ge 40/\eta}} 1 \le \frac{\eta}{2}\log\log{x}.  \]

In this case, $n$ must be divisible by more than $\frac{\eta}{2}\log\log{x}$ primes $\ell$ with ${\Omega(\ell-1)} <40/\eta$. The number of primes $\ell$ up to a given height $T$ satisfying this restriction is $O(T/(\log{T})^{2+o(1)})$, as $T\to\infty$. (In \cite[p. 210]{erdos35}, this estimate is deduced from the upper bound sieve. For more precise results, see \cite{timofeev95}.) In particular, the sum of the reciprocals of such primes $\ell$ is bounded by a certain constant $c$. Thus, the number of possibilities for $n$ is at most
\[ X \sum_{k > \frac{\eta}{2}\log\log{x}} \frac{1}{k!}\bigg(\sum_{\substack{\ell \le X \\ \Omega(\ell-1) < 40/\eta}} \frac{1}{\ell}\bigg)^{k} \le X \sum_{k > \frac{\eta}{2}\log\log{x}} \frac{c^k}{k!}.  \]
(Here we used the multinomial theorem.)
Taking ratios between neighboring terms, we see that the right-hand sum is at most twice its first term (for large $x$). Using Stirling's formula, we find that the right-hand side is crudely bounded above by $x/(\log{x})^{100}$.

Given $n \in (\frac{x}{\log{x}},X]$, the number of corresponding $D$ is
\begin{align*} \ll x \sum_{\lambda \in \Lambda(n)} \frac{1}{\lambda} &\ll \frac{x}{n} (\log\log{x})^2 \cdot \#\Lambda(n) \\ &\ll (\log{x})^2 \cdot \#\Lambda(n) \le (\log{x})^2 \cdot 3^{10\log\log{x}} \ll (\log{x})^{15} . \end{align*}
Summing over the $O(x/(\log{x})^{100})$ possibilities for $n$, we see that only $O(x/(\log{x})^{85})$ values of $D$ arise in Case II.

\vskip 0.1in
\noindent \textbf{Case III: } $\lambda \in \Lambda(n)$ for an $n \in (\frac{x}{\log{x}},X]$ with $\eta \log\log{x} < \omega(n) < 10 \log\log{x}$ and
\[ \sum_{\substack{\ell \mid n \\ \Omega(\ell-1) \ge 40/\eta}} 1 > \frac{\eta}{2}\log\log{x}.  \]

Let $\ell$ be any prime dividing $n$ with $\Omega(\ell-1) \ge 40/\eta$. Choose $\alpha$ with $\ell^{\alpha}\parallel n$. Since $\ell-1$ divides each element of $\Lambda(\ell^{\alpha})$, all of these elements have at least $40/\eta$ prime factors, counted with multiplicity. So from \eqref{eq:lambdadef}, each $\lambda \in \Lambda(n)$ satisfies
\[ \Omega(\lambda) \ge \frac{40}{\eta} \cdot \frac{\eta}{2}\log\log{x} = 20\log\log{x}. \]
In particular, any $D$ divisible by a $\lambda \in \Lambda(n)$  satisfies $\Omega(D) \ge 20\log\log{x}$. But Lemma \ref{lem:HRHT}(ii) implies that the number of such $D \le 12x$ is $O(x/(\log{x})^{10})$.

\vskip 0.1in
\noindent \textbf{Case IV: } $\lambda \in \Lambda(n)$ for an $n \in (\frac{x}{\log{x}},X]$ with $\omega(n) \ge 10 \log\log{x}$.

For each prime $\ell > 2$, we have trivially that $\Omega(\ell-1) \ge 1$. Reasoning as in Case III, we see that each $\lambda \in \Lambda(n)$ satisfies
\[ \Omega(\lambda) \ge \omega(n)-1 > 9\log\log{x}. \]
Thus, any $D$ divisible by such a $\lambda$ also has $\Omega(D) > 9\log\log{x}$. By another application of Lemma \ref{lem:HRHT}(ii), the number of these $D\le 12x$ is $O(x/(\log{x})^5)$.

Assembling the estimates in cases I--IV, we see that the right-hand side of \eqref{eq:ddivvy} is $O(x (\log{x})^{5\epsilon} (\log{x})^{-1})$. Since $\epsilon>0$ is arbitrary, the upper bound is proved.\end{proof}

\section{Proof of Theorem \ref{thm:distinctgroups}: Distribution of maximal torsion subgroups}
\noindent
Here again it is convenient to treat the upper and lower bounds separately. The upper bound uses an elementary and classical mean-value theorem of Wintner.


\begin{prop}[{\cite[Corollary 2.2, p. 50]{SS94}}] \label{prop:wintner} Let $f\colon \Z^+ \ra \mathbb{C}$, and let $g\colon \Z^+ \ra \mathbb{C}$ be determined by the identity
\[ f(n) = \sum_{d\mid n}g(d) \quad\text{for all $n\in \Z^{+}$}. \] If $\sum_{n=1}^{\infty} \frac{|g(n)|}{n} < \infty$,  then as $x\to\infty$,
\[ \sum_{n \le x} f(n) = (\mathfrak{S}+o(1)) x,\quad\text{where}\quad \mathfrak{S}\coloneqq  \sum_{n=1}^{\infty}\frac{g(n)}{n}.\]
Furthermore, if $f$ is multiplicative, then $\mathfrak{S}$ can be written as a convergent Euler product:
\[ \mathfrak{S} = \prod_{p}\left(1+\frac{g(p)}{p} + \frac{g(p^2)}{p^2} + \dots\right). \]
\end{prop}
\noindent
If $G$ is an abelian group of order $n$ and torsion rank at most $2$, then $G$ has a unique representation in the form $\Z/d\Z \times \Z/\frac{n}{d}\Z$, where $d\mid \frac{n}{d}$. So given $n$, the number of such groups $G$ is given by $\tau'(n)\coloneqq \sum_{d^2 \mid n}1$. Notice that $\tau'$ is multiplicative.

In the next lemma, we estimate asymptotically the  number of abelian groups of torsion rank at most $2$ and order at most $y$.

\begin{lem}\label{lem:tau0mean} As $y\to\infty$, we have \[\sum_{n \le y} \tau'(n) \sim \frac{\pi^2}{6}y.\]
\end{lem}
\begin{proof} We apply Proposition \ref{prop:wintner} with $f=\tau_0$ and $g=\mathbf{1}_{\square}$, where $\mathbf{1}_{\square}$ is the characteristic function of the square numbers. Then $\sum_{n=1}^{\infty} \frac{|g(n)|}{n} = \zeta(2) < \infty$. Since $\sum_{n=1}^{\infty}\frac{g(n)}{n} = \zeta(2) = \frac{\pi^2}{6}$, we obtain the lemma.
\end{proof}

\begin{remarks}\mbox{}
\begin{enumerate}
\item[(i)] For each fixed $r \in \Z^{+}$, one can prove in a similar way that the number of abelian groups of order not exceeding $y$ and torsion rank not exceeding $r$ is asymptotic to $(\prod_{2\le k \le r} \zeta(k))y$, as $y\to\infty$. (For a more precise estimate when $r\ge 3$, see \cite{bhowmik93}.) This result dovetails with the theorem of Erd\H{o}s and Szekeres \cite{ES34} that the total number of abelian groups of order at most $y$ is asymptotically $(\prod_{k=2}^{\infty}\zeta(k))y$. Here $\prod_{k=2}^{\infty}\zeta(k)= 2.294856591\dots$.
\item[(ii)] Fix $\alpha > 0$. Proposition \ref{prop:wintner} implies that $\sum_{n \le y} \tau'(n)^{\alpha} \sim \mathfrak{S}_{\alpha}y$, as $y\to\infty$, for some constant $\mathfrak{S}_{\alpha}$. To see this, let $f=\tau'^{\alpha}$, and define $g$ by M\"{o}bius inversion, so that $g(n) = \sum_{d\mid n}\mu(d) \tau'(n/d)^{\alpha}$. In particular, $g(p) = \tau'(p)^{\alpha}-1= 0$, while for prime powers $p^k$ with $k \ge 2$, we have the crude bounds
\[ 0 \le g(p^k) = \tau'(p^k)^{\alpha} -\tau'(p^{k-1})^{\alpha} \le k^{\alpha}.\]
Hence, $\sum_{n=1}^{\infty} \frac{|g(n)|}{n}=\prod_{p}\left(1+\frac{g(p)}{p} + \frac{g(p^2)}{p^2} + \dots\right) = \prod_{p}\left(1+O(\frac{1}{p^2})\right)<\infty$.

We will use this remark below.
\end{enumerate}
\end{remarks}

\begin{proof}[Proof of the upper bound in Theorem \ref{thm:distinctgroups}] From Lemma \ref{lem:tau0mean}, the number of abelian groups of order at most $x/\log{x}$ and torsion rank at most $2$ is $O(x/\log{x})$, which is negligible for our purposes. So it suffices to consider groups that are maximal for degrees $d \le x$ having $T_{\rm CM}(d) > x/\log{x}$. Such $d$ have the property that $T_{\rm CM}'(2d) > x/\log{x}$. Given $\epsilon > 0$, we showed (see \eqref{eq:TCM2dlarge}) that the number of these $d$ is at most $x/(\log{x})^{1-\epsilon}$ for large $x$. Let $\Bb$ be the corresponding set of values of $T_{\rm CM}(d)$. Then the number of maximal torsion subgroups coming from $d$ with $T_{\rm CM}(d) > x/\log{x}$ is at most $\sum_{n \in \Bb} \tau'(n)$. H\"{o}lder's inequality shows that for any positive $\alpha$ and $\beta$ with $\frac{1}{\alpha}+\frac{1}{\beta}=1$,
\begin{align*} \sum_{n \in \Bb} \tau'(n) &\le \Bigg(\sum_{\substack{n \le Cx\log\log{x}}} \tau'(n)^{\alpha}\Bigg)^{1/\alpha} \Bigg(\sum_{n \in \Bb}1 \Bigg)^{1/\beta}.  \end{align*}
Here $C$ has the same meaning as in the proof of Theorem \ref{thm:avg0}(i).
Let $\beta = 1+\epsilon$, so $\alpha = \frac{1+\epsilon}{\epsilon}$. By the second remark following Lemma \ref{lem:tau0mean}, the first sum on $n$ is $O(x\log\log{x})$. The second sum on $n$ is $O(x/(\log{x})^{1-\epsilon})$. So the above right-hand side is
\[ \ll (x\log\log{x})^{\frac{\epsilon}{1+\epsilon}} \cdot x^{\frac{1}{1+\epsilon}} (\log{x})^{-\frac{(1-\epsilon)}{1+\epsilon}} \ll x/(\log{x})^{1-3\epsilon}. \]
Since $\epsilon$ can be taken arbitrarily small, this is acceptable for us.
\end{proof}
\noindent
The lower bound relies on a very recent `anatomical' result of Luca, Pizzarro-Madariaga, and Pomerance.

\begin{prop}[{\cite[Theorem 3]{LPMP15}}]\label{prop:LPMP} There is a $\delta >0$
such that: for all $u \in \Z^+$ and $v \in \Z$, there is $C(u,v) > 0$ such that for all $2 \le z \le x$, the number of primes $\ell \le x$ with $u\ell+v$ having a divisor $p-1$ with $p > z$, $p \ne \ell$, and $p$ prime is at most \[C(u,v) \frac{\pi(x)}{(\log{z})^{\delta}}.\]
\end{prop}

\begin{proof}[Proof of the lower bound in Theorem \ref{thm:distinctgroups}] We will prove the stronger assertion that there are $\gg x/\log{x}$ distinct values of $T_{\rm CM}(d)$ for $d \le x$. We consider degrees $d=\frac{\ell-1}{3}$, where $\ell \in (x/2,x]$ is a prime with $\ell \equiv 1\pmod{3}$. By the prime number theorem for progressions, there are $(\frac{1}{4}+o(1))\frac{x}{\log{x}}$ such primes $\ell$. We will show that for all but $o(x/\log{x})$ of these values of $\ell$, the corresponding $d$ is such that $T_{\rm CM}(d)$ has largest prime factor $\ell$. Consequently, after discarding the $o(x/\log{x})$ exceptional values of $\ell$, we obtain a set of $(\frac{1}{4}+o(1))\frac{x}{\log{x}}$ values of $d$ on which the map $d\mapsto T_{\rm CM}(d)$ is injective.

From Proposition \ref{prop:CCS1A}, there is a CM elliptic curve $E$ over a number field of degree $d$ for which $E$ has a rational point of order $\ell$. So if the largest prime factor of $T_{\rm CM}(d)$ is not $\ell$, then either
\begin{enumerate}
\item[(i)] there is a prime $p$ dividing $T_{\rm CM}(d)$ with $p > \ell$, or
\item[(ii)] $\ell \nmid T_{\rm CM}(d)$ and $T_{\rm CM}(d) > \ell$.
\end{enumerate}
Choose an $F$ of degree $d$ and a CM elliptic curve $E_{/F}$ with $\#E(F)[{\rm tors}] = T_{\rm CM}(d)$. Let $K$ denote the CM field.

In case (i), $E(FK)$ has a point of order $p$. Hence, Lemma \ref{lem:div} implies that
\[ p-1 \mid w_{FK} \frac{[FK:\Q]}{2} \mid w_{FK} d \mid  4(\ell-1). \]
Since $p > \ell > x/2$, Proposition \ref{prop:LPMP} (with $u=4$, $v=-4$) shows that there are only $O(x/(\log{x})^{1+\delta})$ possibilities for $\ell$. This is negligible for us.

Now suppose that we are in case (ii). To start off, we suppose additionally that $\Omega(T_{\rm CM}(d)) > 10\log\log{x}$. Let $n' = \#E(FK)[{\rm tors}]$. Since $T_{\rm CM}(d) = \#E(F)[{\rm tors}] \mid n'$, we have $\Omega(n') >10\log\log{x}$. Theorem \ref{lem:div2} shows that $4(\ell-1)$ is divisible by some $\lambda \in \Lambda(n')$. So from Lemma \ref{lem:lambdadescription}(iii),
\[ \Omega(4(\ell-1)) \ge \Omega(\lambda) \ge \Omega(n')-2 > 9\log\log{x} \]
 (for large $x$). But $4(\ell-1) \le 4x$, and from Lemma \ref{lem:HRHT}(ii) there are only $O(x/(\log{x})^5)$ integers in $[1,4x]$ with more than $9\log\log{x}$ prime factors. In particular, this subcase corresponds to only $o(x/\log{x})$ possible values of $\ell$.

Finally, suppose $\Omega(T_{\rm CM}(d)) < 10\log\log{x}$. Since we are in case (ii), the largest prime factor $r$ of $T_{\rm CM}(d)$ satisfies
\[ r \ge (T_{\rm CM}(d))^{\frac{1}{\Omega(T_{\rm CM}(d))}} > \ell^{1/10\log\log{x}} > z\coloneqq x^{1/20\log\log{x}}. \]
Lemma \ref{lem:div} implies that $r-1 \mid 4\ell-4$. We know also that $r\ne \ell$. Appealing again to Proposition \ref{prop:LPMP}, we find that $\ell$ is restricted to a set of size $O(x (\log\log{x})^{\delta}/(\log{x})^{1+\delta})$. Again, this is negligible.
\end{proof}

\subsection*{Acknowledgments}\noindent We thank Robert S. Rumely for suggesting we investigate prime power Olson degrees. The exposition in \S\ref{sec:average} benefitted from talks by Carl Pomerance on the material in \cite{erdos35}.

The first author was supported in part by NSF grant DMS-1344994 (RTG in Algebra, Algebraic Geometry, and Number Theory, at the University of Georgia).
The third author is supported by NSF award DMS-1402268.

\providecommand{\bysame}{\leavevmode\hbox to3em{\hrulefill}\thinspace}
\providecommand{\MR}{\relax\ifhmode\unskip\space\fi MR }
\providecommand{\MRhref}[2]{%
  \href{http://www.ams.org/mathscinet-getitem?mr=#1}{#2}
}
\providecommand{\href}[2]{#2}

\end{document}